\documentclass[a4paper,reqno]{amsart}

\usepackage{graphicx}
\usepackage{amsmath}
\usepackage{amssymb}
\input xy
\xyoption{all}
\def\del{\delta}

\usepackage{color}

\numberwithin{equation}{section}

\theoremstyle{plain}

\newtheorem{thm}{Theorem}[section]
\newtheorem{cor}[thm]{Corollary}
\newtheorem{lem}[thm]{Lemma}
\newtheorem{prop}[thm]{Proposition}

\newtheorem{defn}[thm]{Definition}
\newtheorem{exm}[thm]{Example}

\theoremstyle{remark}
\newtheorem{rem}[thm]{Remark}

\newdir{ >}{{}*!/-5pt/\dir{>}}

\renewcommand{\mod}{\operatorname{mod}\nolimits}

\newcommand{\Hom}{\operatorname{Hom}\nolimits}

\newcommand{\End}{\operatorname{End}\nolimits}

\newcommand{\Ext}{\operatorname{Ext}\nolimits}

\newcommand{\Cone}{\operatorname{Cone}\nolimits}
\newcommand{\CoCone}{\operatorname{CoCone}\nolimits}

\newcommand{\B}{\mathcal B}

\newcommand{\U}{\mathcal U}
\newcommand{\V}{\mathcal V}
\newcommand{\A}{\mathcal A}

\newcommand{\h}{\mathcal H}
\newcommand{\s}{\mathcal S}
\newcommand{\T}{\mathcal T}

\newcommand{\I}{\mathcal I}
\newcommand{\D}{\mathcal D}
\newcommand{\J}{\mathcal J}
\newcommand{\K}{\mathcal K}

\newcommand{\R}{\mathcal R}
\newcommand{\X}{\mathcal X}

\newcommand{\C}{\mathcal C}

\newcommand{\E}{\mathcal E}
\newcommand{\EE}{\mathbb E}

\newcommand{\svecv}[2]{\left(\begin{smallmatrix}
      #1 \\
      #2
    \end{smallmatrix}\right)}

\newcommand{\svech}[2]{\left(\begin{smallmatrix}
      #1 & #2
\end{smallmatrix}\right)}

\def\Ab{\mathsf{Ab}}
\renewcommand{\emph}{\textit}
\renewcommand{\phi}{\varphi}

\begin{document}

\title{Abelian categories arising from cluster tilting subcategories}\footnote{The first author is supported by the Fundamental Research Funds for the Central Universities (Grants No.2682018ZT25). The second author is supported by the Hunan Provincial Natural Science Foundation of China (Grants No.2018JJ3205) and the NSF of China (Grants No.11671221)}
\author{Yu Liu and Panyue Zhou}
\address{Department of Mathematics, Southwest Jiaotong University, 610031, Chengdu, Sichuan, People's Republic of China}
\email{liuyu86@swjtu.edu.cn}
\address{College of Mathematics, Hunan Institute of Science and Technology, 414006, Yueyang, Hunan, People's Republic of China}
\email{panyuezhou@163.com}
\thanks{The authors wish to thank Professor Bin Zhu for their helpful advices.}

\begin{abstract}
For a triangulated category $\T$, if $\C$ is a cluster-tilting subcategory of $\T$, then the quotient category $\T/\C$ is an abelian category. Under certain conditions, the converse also holds. This is an very important result of cluster-tilting theory, due to Koenig-Zhu and Beligiannis.

Now let $\B$ be a suitable extriangulated category, which is a simultaneous generalization of triangulated categories and exact categories.
We introduce the notion of pre-cluster tilting subcategory $\C$ of
$\B$, which is a generalization of cluster tilting subcategory.
We show that $\C$ is cluster tilting if and only if $\B/\C$ is abelian.
\end{abstract}

\maketitle

\section{Introduction}

Cluster tilting theory gives a way to construct abelian categories from some triangulated categories.
Let $\T$ be a triangulated category and $\C$ a cluster tilting subcategory of $\T$. Then the quotient
category $\T/\C$ is abelian. This is  due to
Koenig and Zhu \cite[Theorem 3.3]{KZ}.
Cluster tilting theory is also permitted to construct abelian categories from some
exact categories. Demonet and Liu \cite[Theorem 3.2]{DL} provided a general framework for passing from exact
categories to abelian categories by factoring out cluster tilting subcategories.

We recall the definition of cluster tilting subcategories, which was introduced by Iyama \cite[Definition 2.2]{I}.
Let $\T$ be a triangulated category or exact category and $\C$ a subcategory of $\T
$. $\C$ is called cluster tilting
if it satisfies:
\begin{itemize}
\item $\C$ is contravariantly finite and covariantly finite;
\item $\C=\C^{\bot_1}={^{\bot_1}}\C$,
where $\C^{\bot_1}=\{M\in\B~|~\Ext^1(\C,M)=0\}$ and ${^{\bot_1}}\C=\{M\in\B~|~\Ext^1(M,\C)=0\}$.
\end{itemize}
Now we consider the opposite direction: if we have an ideal quotient $\T/\D$ which is abelian, can we get any information of $\D$? When does $\D$ become a cluster tilting subcategory?   Beligiannis proved the following characterization of cluster tilting subcategories which
complements, and was inspired by Koenig and Zhu.
\begin{thm}\cite[Theorem 7.3]{B}
Let $\T$ be connected triangulated category  with a Serre functor $\mathbb{S}$ and $\C$ be a non-zero functorially finite rigid subcategory of $\B$. Then the following statements are equivalent:
\begin{itemize}
\item[(a)] $\C$ is cluster tilting;
\item[(b)] $\C$ is a maximal extension closed subcategory of $\T$ such that $\mathbb{S}\C=\C[2]$;
\item[(c)] $\T/\C$ is abelian and $\mathbb{S}\C=\C[2]$.
\end{itemize}
\end{thm}
It is very natural to ask if the similar theory holds on exact category, which also plays an important role in representation theory. For example the extension closed subcategories of module categories of $k$-algebras (where $k$ is a field) are exact categories. Since we usually do not have Serre functors on exact categories, we are also interested in the case for triangulated categories which do not have Serre functors. Hence in this article, we will study a similar case as \cite{B} on a more generalized setting: a category called extriangulated category. The notion of an extriangulated category was introduced in \cite{NP} (please see section 2 for detailed definition of extriangulated category), which is a simultaneous generalization of exact category and triangulated category. For examples of extriangulated categories which are neither exact categories nor triangulated categories, please see \cite{NP,ZZ}. We can also define cluster tilting subcategory on extriangulated categories. Liu and Nakaoka \cite[Theorem 3.2]{LN} showed that any quotient of a extriangulated category modulo a cluster tilting subcategory carried an induced abelian structures, which generalizes both \cite[Theorem 3.3]{KZ} and \cite[Theorem 3.2]{DL}.

In this article, let $(\B,\EE,\mathfrak{s})$ be a Krull-Schmidt extriangulated category over a field $k$. Any subcategory we discuss in this article will be full and closed under isomorphisms.

\begin{rem}
If $\B$ is a $k$-linear, Hom-finite extriangulated category with
split idempotents, then it is a Krull-Schmidt category.
\end{rem}


Now we introduce the notion of pre-cluster tilting subcategory.
\begin{defn}
A subcategory $\B'$ of $\B$ is called \emph{contravariantly finite} if any object in $\B$ admits a right $\B'$-approximation. Moreover, it is called  \emph{strongly contravariantly finite} if any object in $\B$ admits a right $\B'$-approximation which is also a deflation. Dually we can define \emph{strongly covariantly finite}.
\end{defn}

\begin{defn}
We call $\C$ a pre-cluster tilting subcategory of $\B$ if it satisfies the following conditions:
\begin{itemize}
\item $\C$ is closed under direct sums and summands;
\item $\C$ is rigid, that is to say, $\EE(\C,\C)=0$;
\item $\C$ is strongly contravariantly finite and strongly covariantly finite.
\item $\C^{\bot_1}={^{\bot_1}}\C$.
where $\C^{\bot_1}$ is the subcategory of objects $X\in\B$ satisfying
$\EE(\C,X)=0$ and ${^{\bot_1}}\C$ is the subcategory of objects $X\in\B$ satisfying
$\EE(X,\C)=0$.
\end{itemize}
A subcategory $\C$ of $\B$ is called cluster tilting if $\C$ is a pre-cluster tilting and $\C=\C^{\bot_1}={^{\bot_1}}\C$.
\end{defn}

We give an example of pre-cluster tilting subcategory.

\begin{exm}
Let $\Lambda$ be the the $k$-algebra given by the quiver
$$
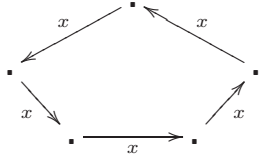
\xymatrix@C=0.5cm@R0.5cm{
&&\centerdot \ar[dll]_x \\
\centerdot \ar[dr]_x &&&&\centerdot \ar[ull]_x\\
&\centerdot \ar[rr]_x &&\centerdot \ar[ur]_x
}
$$
with relation $x^3=0$. Then the AR-quiver of $\B:=\mod\Lambda$ is given by
$$
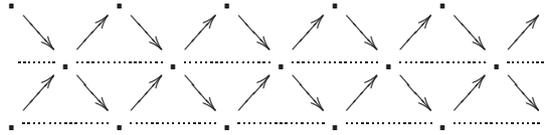
\xymatrix@C=0.4cm@R0.4cm{
\centerdot \ar[dr] &&\centerdot \ar[dr] &&\centerdot \ar[dr] &&\centerdot \ar[dr] &&\centerdot \ar[dr] &&\centerdot \\
\ar@{.}[r] &\centerdot \ar[dr] \ar@{.}[rr] \ar[ur] &&\centerdot \ar@{.}[rr] \ar[dr] \ar[ur] &&\centerdot \ar[dr] \ar[ur] \ar@{.}[rr] &&\centerdot \ar@{.}[rr] \ar[dr] \ar[ur] &&\centerdot \ar[dr] \ar[ur] \ar@{.}[r] &\\
\centerdot \ar@{.}[rr] \ar[ur] &&\centerdot \ar@{.}[rr] \ar[ur] &&\centerdot \ar@{.}[rr] \ar[ur] &&\centerdot \ar@{.}[rr] \ar[ur] &&\centerdot \ar@{.}[rr] \ar[ur] &&\centerdot
}
$$
where the first and the last column are identical. We denote by ``$\circ$" in the AR-quiver the indecomposable objects belong to a subcategory and by ``$\cdot$'' the indecomposable objects do not belong to it. Through direct calculation, we know that 
$$
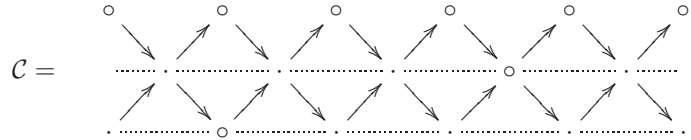
\xymatrix@C=0.4cm@R0.4cm{
&\circ \ar[dr] &&\circ \ar[dr] &&\circ \ar[dr] &&\circ \ar[dr] &&\circ \ar[dr] &&\circ \\
{\C=}  &&\cdot \ar[dr] \ar@{.}[l] \ar@{.}[rr] \ar[ur] &&\cdot \ar@{.}[rr] \ar[dr] \ar[ur] &&\cdot \ar[dr] \ar[ur] \ar@{.}[rr] &&\circ \ar@{.}[rr] \ar[dr] \ar[ur] &&\cdot \ar[dr] \ar[ur] \ar@{.}[r] &\\
&\cdot \ar@{.}[rr] \ar[ur] &&\circ \ar@{.}[rr] \ar[ur] &&\cdot \ar@{.}[rr] \ar[ur] &&\cdot \ar@{.}[rr] \ar[ur] &&\cdot \ar@{.}[rr] \ar[ur] &&\cdot
}
$$
is pre-cluster tilting.
\end{exm}

%

Cotorsion pair is a generalization structure of cluster tilting subcategory on both triangulated and exact categories \cite{N1,DL}, now it is also defined on the extriangulated categories \cite{NP}. We recall its definition, which will be used frequently.
\begin{defn}\cite[Definition 2.1]{NP}
Let $\U$ and $\V$ be two subcategories of $\B$ which are closed under direct summands. We call $(\U,\V)$ a \emph{cotorsion pair} if it satisfies the following conditions:
\begin{itemize}
\item[(a)] $\EE(\U,\V)=0$.

\item[(b)] For any object $B\in \B$, there exist two $\EE$-triangles
\begin{align*}
V_B\rightarrow U_B\rightarrow B\overset{\delta}{\dashrightarrow},\quad
B\rightarrow V^B\rightarrow U^B\overset{\sigma}{\dashrightarrow}
\end{align*}
satisfying $U_B,U^B\in \U$ and $V_B,V^B\in \V$.
\end{itemize}
\end{defn}
By definition of a cotorsion pair, we can immediately conclude:
\begin{prop}\label{prop0}
Let $\U$ be a subcategory of $\B$. Then $(\U,\U)$ is a cotorsion pair
if and only if $\U$ is a cluster tilting subcategory.
\end{prop}

For any subcategory $\U$, we call $\U\cap {^{\bot_1}}\U=$ the coheart of $\U$. We say $\U$ is maximal if $\U$ is maximal among those with the same coheart.
\vspace{1mm}

Now we assume $\B$ has enough projectives and enough injectives. We denote by $\mathcal P$ the subcategory of projective objects and by $\mathcal I$ the subcategory of injective objects. Under this assumption, if we have a pre-cluster tilting subcategory $\C$, then we can get two cotorsion pairs $(\C,\C^{\bot_1})$ and $(\C^{\bot_1},\C)$ (see Lemma \ref{CP}).

Our first main result is the following.

\begin{thm}\label{main1}
The maps
$$\U \mapsto \C:=\U\cap {^{\bot_1}}\U \text{ and } \C \mapsto \U:={^{\bot_1}}\C $$
give mutually inverse bijections between:
\begin{itemize}
\item Maximal subcategories $\U$ which admits two cotorsion pairs $(\U,\V),(\V,\U)$.
\item Pre-cluster tilting subcategories $\C$.
\end{itemize}
\end{thm}

In order to introduce the second main result, we need the following definition.

\begin{defn}
Let $\B_j,j\in \J$
be extriangulated subcategories of $\B$. We call that $\B$ is  a direct sum of
extriangulated subcategories $\B_j,j\in\J$ if it satisfies the following conditions:
\begin{itemize}
\item Any object $M\in\B$ is a direct sum of finitely many objects $M_j\in\B_j$;
\item $\Hom(\B_i,\B_j)=0$, for any $i\neq j$.
\end{itemize}
In this case, we write $\B=\bigoplus_{j\in\J}\B_j$.
An extriangulated category is called connected if it can not be written as direct sum of two non-zero extriangulated subcategories.
\end{defn}

By \cite[Proposition 3.30]{NP}, $\B/(\mathcal P\cap \mathcal I)$ is still an extriangulated category. We will show the second main result of this article.

\begin{thm}\label{main}
Let $\B/(\mathcal P\cap \mathcal I)$ be connected and $(\U,\V),(\V,\U)$ be cotorsion pairs on $\B$. Let $\C=\U\cap \V$, if $\C\supset \mathcal P\cap \mathcal I$, then the following statements are equivalent.
\begin{itemize}
\item[(a)] $\C$ is cluster tilting;
\item[(b)] If $(\s,\R),(\R,\s)$ are cotorsion pairs such that $\s\cap \R=\C$, then $\s=\C$;
\item[(c)] $\B/\C$ is abelian.
\item[(d)]  $\B/\U$ and $\B/\V$ are abelian.
\end{itemize}
\end{thm}

Theorem \ref{main1} and Theorem \ref{main} is a generalization of related results of Koenig-Zhu \cite[Theorem 3.3]{KZ}, Demonet-Liu \cite[Theorem 3.2]{DL} and Beligiannis \cite[Theorem 7.3]{B}.

This article is organized as follows. In Section 2, we review some elementary definitions and facts of extriangulated category
that we need. In Section 3, we prove our first and second main result.

\section{Preliminaries}
Let us briefly recall the definition and basic properties of extriangulated categories from \cite{NP}. Throughout this paper, we assume that $\B$ is an additive category.

\begin{defn}
Suppose that $\B$ is equipped with an additive bifunctor $\mathbb{E}\colon\B^\mathrm{op}\times\B\to\Ab$, where $\Ab$ is the category of abelian groups. For any pair of objects $A,C\in\B$, an element $\delta\in\mathbb{E}(C,A)$ is called an {\it $\mathbb{E}$-extension}. Thus formally, an $\EE$-extension is a triplet $(A,\delta,C)$.
For any $A,C\in\C$, the zero element $0\in\EE(C,A)$ is called the \emph{spilt $\EE$-extension}.

Let $\delta\in\mathbb{E}(C,A)$ be any $\mathbb{E}$-extension. By the functoriality, for any $a\in\B(A,A^{\prime})$ and $c\in\B(C^{\prime},C)$, we have $\mathbb{E}$-extensions
\[ \mathbb{E}(C,a)(\delta)\in\mathbb{E}(C,A^{\prime})\ \ \text{and}\ \ \mathbb{E}(c,A)(\delta)\in\mathbb{E}(C^{\prime},A). \]
We abbreviately denote them by $a_{\ast}\delta$ and $c^{\ast}\delta$.
In this terminology, we have
\[ \mathbb{E}(c,a)(\delta)=c^{\ast} a_{\ast}\delta=a_{\ast} c^{\ast}\delta \]
in $\mathbb{E}(C^{\prime},A^{\prime})$.

\end{defn}

\begin{defn}
Let $\delta\in\mathbb{E}(C,A)$ and $\delta^{\prime}\in\mathbb{E}(C^{\prime},A^{\prime})$ be two pair of $\mathbb{E}$-extensions. A {\it morphism} $(a,c)\colon\delta\to\delta^{\prime}$ of $\mathbb{E}$-extensions is a pair of morphisms $a\in\B(A,A^{\prime})$ and $c\in\B(C,C^{\prime})$ in $\B$, satisfying the equality
\[ a_{\ast}\delta=c^{\ast}\delta^{\prime}. \]
We simply denote it as $(a,c)\colon\delta\to\delta^{\prime}$.
\end{defn}

\begin{defn}
Let $\delta=(A,\delta,C)$ and $\delta^{\prime}=(A^{\prime},\delta^{\prime},C^{\prime})$ be any pair of $\mathbb{E}$-extensions. Let
\[ C\xrightarrow{~\iota_C~}C\oplus C^{\prime}\xleftarrow{~\iota_{C^{\prime}}~}C^{\prime} \]
and
\[ A\xrightarrow{~p_A~}A\oplus A^{\prime}\xleftarrow{~p_{A^{\prime}}~}A^{\prime} \]
be coproduct and product in $\B$, respectively. Remark that, by the additivity of $\mathbb{E}$, we have a natural isomorphism
\[ \mathbb{E}(C\oplus C^{\prime},A\oplus A^{\prime})\simeq \mathbb{E}(C,A)\oplus\mathbb{E}(C,A^{\prime})\oplus\mathbb{E}(C^{\prime},A)\oplus\mathbb{E}(C^{\prime},A^{\prime}). \]

Let $\delta\oplus\delta^{\prime}\in\mathbb{E}(C\oplus C^{\prime},A\oplus A^{\prime})$ be the element corresponding to $(\delta,0,0,\delta^{\prime})$ through this isomorphism. This is the unique element which satisfies
$$
\mathbb{E}(\iota_C,p_A)(\delta\oplus\delta^{\prime})=\delta,\ \mathbb{E}(\iota_C,p_{A^{\prime}})(\delta\oplus\delta^{\prime})=0,\
\mathbb{E}(\iota_{C^{\prime}},p_A)(\delta\oplus\delta^{\prime})=0,\ \mathbb{E}(\iota_{C^{\prime}},p_{A^{\prime}})(\delta\oplus\delta^{\prime})=\delta^{\prime}.
$$
\end{defn}

\begin{defn}
Let $A,C\in\B$ be any pair of objects. Two sequences of morphisms in $\B$
\[ A\overset{x}{\longrightarrow}B\overset{y}{\longrightarrow}C\ \ \text{and}\ \ A\overset{x^{\prime}}{\longrightarrow}B^{\prime}\overset{y^{\prime}}{\longrightarrow}C \]
are said to be {\it equivalent} if there exists an isomorphism $b\in\B(B,B^{\prime})$ which makes the following diagram commutative.
\[
\xy
(-16,0)*+{A}="0";
(3,0)*+{}="1";
(0,8)*+{B}="2";
(0,-8)*+{B^{\prime}}="4";
(-3,0)*+{}="5";
(16,0)*+{C}="6";
{\ar^{x} "0";"2"};
{\ar^{y} "2";"6"};
{\ar_{x^{\prime}} "0";"4"};
{\ar_{y^{\prime}} "4";"6"};
{\ar^{b}_{\simeq} "2";"4"};
{\ar@{}|{} "0";"1"};
{\ar@{}|{} "5";"6"};
\endxy
\]

We denote the equivalence class of $A\overset{x}{\longrightarrow}B\overset{y}{\longrightarrow}C$ by $[A\overset{x}{\longrightarrow}B\overset{y}{\longrightarrow}C]$.
\end{defn}

\begin{defn}
$\ \ $
\begin{enumerate}
\item[(1)] For any $A,C\in\B$, we denote as
\[ 0=[A\overset{\Big[\raise1ex\hbox{\leavevmode\vtop{\baselineskip-8ex \lineskip1ex \ialign{#\crcr{$\scriptstyle{1}$}\crcr{$\scriptstyle{0}$}\crcr}}}\Big]}{\longrightarrow}A\oplus C\overset{[0\ 1]}{\longrightarrow}C]. \]

\item[(2)] For any $[A\overset{x}{\longrightarrow}B\overset{y}{\longrightarrow}C]$ and $[A^{\prime}\overset{x^{\prime}}{\longrightarrow}B^{\prime}\overset{y^{\prime}}{\longrightarrow}C^{\prime}]$, we denote as
\[ [A\overset{x}{\longrightarrow}B\overset{y}{\longrightarrow}C]\oplus [A^{\prime}\overset{x^{\prime}}{\longrightarrow}B^{\prime}\overset{y^{\prime}}{\longrightarrow}C^{\prime}]=[A\oplus A^{\prime}\overset{x\oplus x^{\prime}}{\longrightarrow}B\oplus B^{\prime}\overset{y\oplus y^{\prime}}{\longrightarrow}C\oplus C^{\prime}]. \]
\end{enumerate}
\end{defn}

\begin{defn}
Let $\mathfrak{s}$ be a correspondence which associates an equivalence class $\mathfrak{s}(\delta)=[A\overset{x}{\longrightarrow}B\overset{y}{\longrightarrow}C]$ to any $\mathbb{E}$-extension $\delta\in\mathbb{E}(C,A)$. This $\mathfrak{s}$ is called a {\it realization} of $\mathbb{E}$, if it satisfies the following condition $(\star)$. In this case, we say that the sequence $A\overset{x}{\longrightarrow}B\overset{y}{\longrightarrow}C$ {\it realizes} $\delta$, whenever it satisfies $\mathfrak{s}(\delta)=[A\overset{x}{\longrightarrow}B\overset{y}{\longrightarrow}C]$.
\begin{itemize}
\item[$(\star)$] Let $\delta\in\mathbb{E}(C,A)$ and $\delta^{\prime}\in\mathbb{E}(C^{\prime},A^{\prime})$ be any pair of $\mathbb{E}$-extensions, with
\[\mathfrak{s}(\delta)=[A\overset{x}{\longrightarrow}B\overset{y}{\longrightarrow}C]\text{ and } \mathfrak{s}(\delta^{\prime})=[A^{\prime}\overset{x^{\prime}}{\longrightarrow}B^{\prime}\overset{y^{\prime}}{\longrightarrow}C^{\prime}].\]
Then, for any morphism $(a,c)\colon\delta\to\delta^{\prime}$, there exists $b\in\B(B,B^{\prime})$ which makes the following diagram commutative.
$$
\xy
(-12,6)*+{A}="0";
(0,6)*+{B}="2";
(12,6)*+{C}="4";
(-12,-6)*+{A^{\prime}}="10";
(0,-6)*+{B^{\prime}}="12";
(12,-6)*+{C^{\prime}}="14";
{\ar^{x} "0";"2"};
{\ar^{y} "2";"4"};
{\ar_{a} "0";"10"};
{\ar^{b} "2";"12"};
{\ar^{c} "4";"14"};
{\ar^{x^{\prime}} "10";"12"};
{\ar^{y^{\prime}} "12";"14"};
{\ar@{}|{} "0";"12"};
{\ar@{}|{} "2";"14"};
\endxy
$$
\end{itemize}
In the above situation, we say that the triplet $(a,b,c)$ {\it realizes} $(a,c)$.
\end{defn}

\begin{defn}
Let $\B,\mathbb{E}$ be as above. A realization of $\mathbb{E}$ is said to be {\it additive}, if it satisfies the following conditions.
\begin{itemize}
\item[{\rm (i)}] For any $A,C\in\B$, the split $\mathbb{E}$-extension $0\in\mathbb{E}(C,A)$ satisfies
\[ \mathfrak{s}(0)=0. \]
\item[{\rm (ii)}] For any pair of $\mathbb{E}$-extensions $\delta\in\mathbb{E}(C,A)$ and $\delta^{\prime}\in\mathbb{E}(C^{\prime},A^{\prime})$, we have
\[ \mathfrak{s}(\delta\oplus\delta^{\prime})=\mathfrak{s}(\delta)\oplus\mathfrak{s}(\delta^{\prime}). \]
\end{itemize}
\end{defn}

\begin{defn}\cite[Definition 2.12]{NP}
A triplet $(\B,\mathbb{E},\mathfrak{s})$ is called an {\it extriangulated category} if it satisfies the following conditions.
\begin{itemize}
\item[{\rm (ET1)}] $\mathbb{E}\colon\B^{\mathrm{op}}\times\B\to\Ab$ is an additive bifunctor.
\item[{\rm (ET2)}] $\mathfrak{s}$ is an additive realization of $\mathbb{E}$.
\item[{\rm (ET3)}] Let $\delta\in\mathbb{E}(C,A)$ and $\delta^{\prime}\in\mathbb{E}(C^{\prime},A^{\prime})$ be any pair of $\mathbb{E}$-extensions, realized as
\[ \mathfrak{s}(\delta)=[A\overset{x}{\longrightarrow}B\overset{y}{\longrightarrow}C],\ \ \mathfrak{s}(\delta^{\prime})=[A^{\prime}\overset{x^{\prime}}{\longrightarrow}B^{\prime}\overset{y^{\prime}}{\longrightarrow}C^{\prime}]. \]
For any commutative square
$$
\xy
(-12,6)*+{A}="0";
(0,6)*+{B}="2";
(12,6)*+{C}="4";
(-12,-6)*+{A^{\prime}}="10";
(0,-6)*+{B^{\prime}}="12";
(12,-6)*+{C^{\prime}}="14";
{\ar^{x} "0";"2"};
{\ar^{y} "2";"4"};
{\ar_{a} "0";"10"};
{\ar^{b} "2";"12"};
{\ar^{x^{\prime}} "10";"12"};
{\ar^{y^{\prime}} "12";"14"};
{\ar@{}|{} "0";"12"};
\endxy
$$
in $\B$, there exists a morphism $(a,c)\colon\delta\to\delta^{\prime}$ satisfying $cy=y^{\prime}b$.
\item[{\rm (ET3)$^{\mathrm{op}}$}] Dual of {\rm (ET3)}.
\item[{\rm (ET4)}] Let $\delta\in\mathbb{E}(D,A)$ and $\delta^{\prime}\in\mathbb{E}(F,B)$ be $\mathbb{E}$-extensions realized by
\[ A\overset{f}{\longrightarrow}B\overset{f^{\prime}}{\longrightarrow}D\ \ \text{and}\ \ B\overset{g}{\longrightarrow}C\overset{g^{\prime}}{\longrightarrow}F \]
respectively. Then there exist an object $E\in\B$, a commutative diagram
$$
\xy
(-21,7)*+{A}="0";
(-7,7)*+{B}="2";
(7,7)*+{D}="4";
(-21,-7)*+{A}="10";
(-7,-7)*+{C}="12";
(7,-7)*+{E}="14";
(-7,-21)*+{F}="22";
(7,-21)*+{F}="24";
{\ar^{f} "0";"2"};
{\ar^{f^{\prime}} "2";"4"};
{\ar@{=} "0";"10"};
{\ar_{g} "2";"12"};
{\ar^{d} "4";"14"};
{\ar^{h} "10";"12"};
{\ar^{h^{\prime}} "12";"14"};
{\ar_{g^{\prime}} "12";"22"};
{\ar^{e} "14";"24"};
{\ar@{=} "22";"24"};
{\ar@{}|{} "0";"12"};
{\ar@{}|{} "2";"14"};
{\ar@{}|{} "12";"24"};
\endxy
$$
in $\B$, and an $\mathbb{E}$-extension $\delta^{\prime\prime}\in\mathbb{E}(E,A)$ realized by $A\overset{h}{\longrightarrow}C\overset{h^{\prime}}{\longrightarrow}E$, which satisfy the following compatibilities.
\begin{itemize}
\item[{\rm (i)}] $D\overset{d}{\longrightarrow}E\overset{e}{\longrightarrow}F$ realizes $f^{\prime}_{\ast}\delta^{\prime}$,
\item[{\rm (ii)}] $d^{\ast}\delta^{\prime\prime}=\delta$,

\item[{\rm (iii)}] $f_{\ast}\delta^{\prime\prime}=e^{\ast}\delta^{\prime}$.
\end{itemize}

\item[{\rm (ET4)$^{\mathrm{op}}$}] Dual of {\rm (ET4)}.
\end{itemize}
\end{defn}

\begin{rem}
Note that both exact categories and triangulated categories are extriangulated categories, see \cite[Example 2.13]{NP} and extension-closed subcategories of extriangulated categories are
again extriangulated, see \cite[Remark 2.18]{NP} . Moreover, there exist extriangulated categories which
are neither exact categories nor triangulated categories, see \cite[Proposition 3.30]{NP} and \cite[Example 4.14]{ZZ}.
\end{rem}

We will use the following terminology.
\begin{defn}{\cite{NP}}
Let $(\B,\EE,\mathfrak{s})$ be an extriangulated category.
\begin{itemize}
\item[(1)] A sequence $A\xrightarrow{~x~}B\xrightarrow{~y~}C$ is called a {\it conflation} if it realizes some $\EE$-extension $\del\in\EE(C,A)$. In this case, $x$ is called an {\it inflation} and $y$ is called a {\it deflation}.

\item[(2)] If a conflation  $A\xrightarrow{~x~}B\xrightarrow{~y~}C$ realizes $\delta\in\mathbb{E}(C,A)$, we call the pair $( A\xrightarrow{~x~}B\xrightarrow{~y~}C,\delta)$ an {\it $\EE$-triangle}, and write it in the following way.
$$A\overset{x}{\longrightarrow}B\overset{y}{\longrightarrow}C\overset{\delta}{\dashrightarrow}$$
We usually do not write this $``\delta"$ if it is not used in the argument.
\item[(3)] Let $A\overset{x}{\longrightarrow}B\overset{y}{\longrightarrow}C\overset{\delta}{\dashrightarrow}$ and $A^{\prime}\overset{x^{\prime}}{\longrightarrow}B^{\prime}\overset{y^{\prime}}{\longrightarrow}C^{\prime}\overset{\delta^{\prime}}{\dashrightarrow}$ be any pair of $\EE$-triangles. If a triplet $(a,b,c)$ realizes $(a,c)\colon\delta\to\delta^{\prime}$, then we write it as
$$\xymatrix{
A \ar[r]^x \ar[d]^a & B\ar[r]^y \ar[d]^{b} & C\ar@{-->}[r]^{\del}\ar[d]^c&\\
A'\ar[r]^{x'} & B' \ar[r]^{y'} & C'\ar@{-->}[r]^{\del'} &}$$
and call $(a,b,c)$ a {\it morphism of $\EE$-triangles}.

\item[(4)] An object $P\in\B$ is called {\it projective} if
for any $\EE$-triangle $A\overset{x}{\longrightarrow}B\overset{y}{\longrightarrow}C\overset{\delta}{\dashrightarrow}$ and any morphism $c\in\B(P,C)$, there exists $b\in\B(P,B)$ satisfying $yb=c$.
We denote the subcategory of projective objects by $\mathcal P\subseteq\B$. Dually, the subcategory of injective objects is denoted by $\I\subseteq\B$.

\item[(5)] We say that $\B$ {\it has enough projective objects} if
for any object $C\in\B$, there exists an $\EE$-triangle
$A\overset{x}{\longrightarrow}P\overset{y}{\longrightarrow}C\overset{\delta}{\dashrightarrow}$
satisfying $P\in\mathcal P$. Dually we can define $\B$ {\it has enough injective objects}.

\item[(6)] Let $\mathcal{X}$ be a subcategory of $\B$. We say $\mathcal{X}$ is {\it extension-closed}
if a conflation $A\rightarrowtail B\twoheadrightarrow C$ satisfies $A,C\in\mathcal{X}$, then $B\in\mathcal{X}$.
\end{itemize}
\end{defn}

In this article, we always assume $\B$ has enough projectives and enough injectives.
\medskip

By \cite{NP}, we give the following useful remark, which will be used later in the proofs.

\begin{rem}\label{useful}
Let $\xymatrix{A\ar[r]^a &B \ar[r]^b &C \ar@{-->}[r] &}$ and $\xymatrix{X\ar[r]^x &Y \ar[r]^y &Z \ar@{-->}[r] &}$ be two $\EE$-triangles. Then
\begin{itemize}
\item[(a)] In this following commutative diagram
$$\xymatrix{
X\ar[r]^x \ar[d]_f &Y \ar[d]^g \ar[r]^y &Z \ar[d]^h \ar@{-->}[r] &\\
A\ar[r]^a &B \ar[r]^b &C \ar@{-->}[r] &}
$$
$f$ factors through $x$ if and only if $h$ factors through $b$.
\item[(b)] In the following commutative diagram
$$\xymatrix{
A\ar[r]^a \ar[d]_s &B \ar[d]^r \ar[r]^b &C \ar[d]^t \ar@{-->}[r] &\\
X\ar[r]^x \ar[d]_f &Y \ar[d]^g \ar[r]^y &Z \ar[d]^h \ar@{-->}[r] &\\
A\ar[r]^a &B \ar[r]^b &C \ar@{-->}[r] &}
$$
$fs=1_A$ implies $B$ is a direct summand of $C\oplus Y$ and $C$ is a direct summand of $Z\oplus B$; $ht=1_C$ implies $B$ is a direct summand of $A\oplus Y$ and $A$ is a direct summand of $X\oplus B$.
\item[(c)] If we have $b:B\xrightarrow{d_1} D \xrightarrow{d_2} C$ and $d_2: D\xrightarrow{d_3} B\xrightarrow{b} C$, then $B$ is a direct summand of $A\oplus D$.\\
If we have $a: A\xrightarrow{e_1} E\xrightarrow{e_2} B$ and $e_1:A \xrightarrow{a} B\xrightarrow{e_3} E$, then $B$ is a direct summand of $C\oplus E$.
\end{itemize}
\end{rem}

We first recall the following proposition (\cite[Proposition 1.20]{LN}), which (also the dual of it) will be used many times in the article.

\begin{prop}
Let $A\overset{x}{\longrightarrow}B\overset{y}{\longrightarrow}C\overset{\delta}{\dashrightarrow}$ be any $\EE$-triangle, let $f\colon A\rightarrow D$ be any morphism, and let $D\overset{d}{\longrightarrow}E\overset{e}{\longrightarrow}C\overset{f_{\ast}\delta}{\dashrightarrow}$ be any $\EE$-triangle realizing $f_{\ast}\delta$. Then there is a morphism $g$ which gives a morphism of $\EE$-triangles
$$\xymatrix{
A \ar[r]^{x} \ar[d]_f &B \ar[r]^{y} \ar[d]^g &C \ar@{=}[d]\ar@{-->}[r]^{\delta}&\\
D \ar[r]_{d} &E \ar[r]_{e} &C\ar@{-->}[r]_{f_{\ast}\delta}&
}
$$
and moreover, the sequence $A\overset{\svecv{f}{x}}{\longrightarrow}D\oplus B\overset{\svech{d}{-g}}{\longrightarrow}E\overset{e^{\ast}\delta}{\dashrightarrow}$ becomes an $\EE$-triangle.
\end{prop}

We prove the following lemma related to cotorsion pairs.

\begin{lem}\label{CP}
If $\C$ is rigid and strongly contravariantly finite, then $(\C,\C^{\bot_1})$ is a cotorsion pair.
\end{lem}

\begin{proof}
Since $\B$ has enough injectives, any object $A\in\B$ admits an $\EE$-triangle $\xymatrix{A \ar[r] &I \ar[r] &B \ar@{-->}[r] &}$ where $I$ is injective. Since $\C$ is strongly contravariantly finite, the object $B$ admits a conflation $$\xymatrix{B_1 \ar[r] &C_0 \ar[r]^{f_0} &B \ar@{-->}[r] &}$$ where $f_0$ is a right $\C$-approximation of $B$. The rigidity of $\C$ implies $B_1\in \C^{\bot_1}$. We have the following commutative diagram
$$\xymatrix{
&A \ar[d] \ar@{=}[r] &A \ar[d]\\
B_1 \ar@{=}[d] \ar[r] &X \ar[r] \ar[d] &I \ar@{-->}[r]\ar[d]&\\
B_1 \ar[r] &C_0\ar@{-->}[d] \ar[r]^{f_0} &B\ar@{-->}[d]\ar@{-->}[r]&\\
&&}
$$
where $X\in \C^{\bot_1}$. Hence by definition, the pair $(\C,\C^{\bot_1})$ is a cotorsion pair.
\end{proof}

\section{Main results}
Let $\A$ be an additive category and $\X$ be a subcategory of $\A$. We denote by $\A/\X$
the category whose objects are objects of $\A$ and whose morphisms are elements of
$\Hom_{\A}(A,B)/\X(A,B)$ for $A,B\in\A$, where $\X(A,B)$ the subgroup of $\Hom_{\A}(A,B)$ consisting of morphisms
which factor through an object in $\X$. Such category is called the quotient category
of $\A$ by $\X$. For any morphism $f\colon A\to B$ in $\A$, we denote by $\overline{f}$ the image of $f$ under
the natural quotient functor $\A\to\A/\X$.

\vspace{1mm}

We first introduce some notions.

Let $\B'$ and $\B''$ be two subcategories of $\B$, denote by $\CoCone(\B',\B'')$ the subcategory of objects $X$ admitting an $\EE$-triangle $\xymatrix{X \ar[r] &B'\ar[r] &B''\ar@{-->}[r] &}$ where $B'\in \B'$ and $B''\in \B''$. We denote by $\Cone(\B',\B'')$ the subcategory of objects $Y$ admitting an $\EE$-triangle $\xymatrix{B' \ar[r] &B''\ar[r] &Y\ar@{-->}[r] &}$ where $B'\in \B'$ and $B''\in \B''$.

Let $\Omega \B'=\CoCone(\mathcal P,\B')$ and $\Sigma \B'=\Cone(\B',\mathcal I)$. We write an object $D$ in the form $\Omega B$ if it admits an $\EE$-triangle $\xymatrix{D \ar[r] &P \ar[r] &B \ar@{-->}[r] &}$ where $P\in \mathcal P$. We write an object $D'$ in the form $\Sigma B'$ if it admits an $\EE$-triangle $\xymatrix{B' \ar[r] &I \ar[r] &D' \ar@{-->}[r] &}$ where $I\in \mathcal I$.

In the rest of this article, let $(\U,\V),(\V,\U)$ be cotorsion pairs, we denote $\U\cap \V$ by $\C$. We denote subcategory $\{ \text{direct sums of objects in }\U \text{ and objects in }\V\}$ by $\K$, and we say $\K=\U+\V$. Let $\h=\CoCone(\C,\U)\cap \Cone(\V,\C)$, $\h/\C$ is called the heart of $(\U,\V)$, it is abelian by \cite{LN}. Let $H$ be the cohomological functor defined in \cite{LN}, it sends an $\EE$-triangle $\xymatrix{A \ar[r]^f &B \ar[r]^g &C \ar@{-->}[r] &}$ to an exact sequence $H(A)\xrightarrow{H(f)} H(B)\xrightarrow{H(g)} H(C)$ in $\h/\C$, moreover, $H(B)=0$ if and only if $B\in \K$.

Since $\U,\V$ are extension closed subcategories of $\B$, they are extriangulated subcategories. Moreover, $\C$ is the subcategory of enough projective-injective objects in $\U$ and $\V$, according to \cite{NP}, $\U/\C$ and $\V/\C$ are triangulated categories.

\begin{lem}\label{direct summand}
$\h$ is closed under direct summands.
\end{lem}

\begin{proof}
Since $\h=\CoCone(\C,\U)\cap \Cone(\V,\C)$, we will show that $\CoCone(\C,\U)$ is closed under direct summands. By dual, we can show that $\Cone(\V,\C)$ is closed under direct summands.

Assume we have an $\EE$-triangle $$\xymatrix{X\oplus Y \ar[r]^-{\svech{x}{y}} &C_1 \ar[r] &U_2 \ar@{-->}[r] &}$$ where $C_1\in \C$ and $\U_2\in \U$, then $x$ is an inflation and it admits an $\EE$-triangle $$\xymatrix{X \ar[r]^-{x} &C_1 \ar[r]^c &U \ar@{-->}[r] &}.$$ Since $\EE(U_2,C_1)=0$, there is a morphism $f\colon C_1\to C_1$ such that
$f(x\ \ y)=(x\ \ 0)$. In particular, we have $fx=x$ and $fy=0$. Hence we have the following commutative diagram
$$\xymatrix{
X \ar[r]^-{x} \ar[d]_-{\svecv{1}{0}} &C_1 \ar@{=}[d] \ar[r]^c &U \ar[d]^a \ar@{-->}[r] &\\
X\oplus Y \ar[r]^-{\svech{x}{y}} \ar[d]_-{\svech{1}{0}} &C_1 \ar[d]^f \ar[r] &U_2 \ar[d]^b \ar@{-->}[r] &\\
X \ar[r]^-{x} &C_1 \ar[r]^c &U \ar@{-->}[r] &.
}
$$
It follows that there exists a morphism $d:U\rightarrow C_1$ such that $1_{U}-ba=cd$.
Hence $U$ is a direct summand of $C_1\oplus U_2\in \U$ and then $U\in \U$. This implies $X\in \CoCone(\C,\U)$.
\end{proof}

\textbf{Proof of Theorem \ref{main1}:}

\begin{proof}
We show $(\C,\K)$ is a cotorsion pair, then dually $(\K,\C)$ is also a cotorsion pair.\\
Let $B$ be an object in $\B$, since $(\V,\U)$ is a cotorsion pair, $B$ admits a commutative diagram of $\EE$-triangles
$$\xymatrix{
\Omega U_B \ar[r] \ar[d]_f &P \ar[r] \ar[d] &U_B \ar@{=}[d] \ar@{-->}[r] &\\
B \ar[r] &V_B \ar[r] &U_B \ar@{-->}[r] &
}$$
where $U_B\in \U$ and $V_B$ and $P\in \mathcal P$. We get $H(f)$ is an epimorphism. $\Omega U_B$ admits the following commutative diagram
$$\xymatrix{
\Omega U_B \ar[r] \ar[d] &P \ar[r] \ar[d] &U_B \ar@{=}[d] \ar@{-->}[r] &\\
U \ar[r] \ar[d] &C \ar[r] \ar[d] &U_B \ar@{-->}[r] &\\
V \ar@{=}[r] \ar@{-->}[d]^{\sigma} &V \ar@{-->}[d]\\
&&&&
}
$$
where $U\in \U$ and $V\in \V$. From the second column we get $C= P\oplus V\in \V$. From the second row we get $C\in \U$, hence $C\in \C$ and $V$ is a direct summand of $C$, then $V\in \C$. Now $B$ admits the following commutative diagram:
$$\xymatrix{
\Omega U_B \ar[r] \ar[d]_f &U \ar[r] \ar[d] &V \ar@{=}[d] \ar@{-->}[r]^{\sigma} &\\
B \ar[r]_g &K \ar[r]_h &V \ar@{-->}[r]^{f_*\sigma} &
}
$$
By applying $H$ we get the following exact sequence $H(\Omega B)\xrightarrow{H(f)} H(B)\xrightarrow{H(g)} H(K) \xrightarrow{H(h)} H(V)=0$. Since $H(f)$ is an epimorphism, we have $H(g)=0$, hence $H(K)=0$, which implies $K\in \K$. $B$ admits an $\EE$-triangle $\xymatrix{\Omega B \ar[r]  &P_B \ar[r] &B \ar@{-->}[r] &}$ where $P_B\in \mathcal P$, by the previous argument, $\Omega B$ admits an $\EE$-triangle $\xymatrix{\Omega B \ar[r] &K' \ar[r] &C' \ar@{-->}[r] &}$ where $K'\in \K$ and $C'\in \C$, hence we get the following commutative diagram\\
$$\xymatrix{
\Omega B \ar[d] \ar[r] &K' \ar[d] \ar[r] &C' \ar@{=}[d] \ar@{-->}[r] &\\
P_B \ar[r] \ar[d] &P_B\oplus C' \ar[r] \ar[d] & C' \ar@{-->}[r] &\\
B \ar@{=}[r] \ar@{-->}[d] &B \ar@{-->}[d] \\
&& &&
}
$$
Since $P_B\oplus C'\in \C$ and $\EE(\C,\K)=0$, by definition $(\C,\K)$ is a cotorsion pair.\\
Hence $\C$ is pre-cluster tilting. Now if $\U$ is maximal, $\U$ has to be $\K$.\\
On the other hand, if $\C$ is pre-cluster tilting, by Lemma \ref{CP} and its dual, we have cotorsion pairs $(\K,\C),(\C,\K)$ where $\K$ is maximal by the previous argument.
\end{proof}

\begin{lem}\label{serre}
 Let $A\xrightarrow{ f} B$ be a morphism in $\B$ such that $A,B$ do not have direct summand in $\V$. Then
\begin{itemize}
\item[(a)]  If $\overline f$ is a monomorphism in $\B/\V$ and $B\in \U$, then $A\in \U$.
\item[(b)]  If $\overline f$ is an epimorphism in $\B/\V$ and $A\in \U$, then  $B\in \U$.
\end{itemize}
\end{lem}

\begin{proof}
Since $A$ admits an $\EE$-triangle $\xymatrix{A\ar[r] &V^A \ar[r] &U^A \ar@{-->}[r] &}$ where $V^A\in \V$ and $U^A\in \U$, and $B$ admits an $\EE$-triangle $\xymatrix{ U_B \ar[r] &V_B \ar[r] &B \ar@{-->}[r] &}$ where $U_B\in \U$ and $V_B\in \V$, we have the following commutative diagram
$$\xymatrix{
U_B \ar@{=}[d] \ar[r] &K' \ar[r]^{g'} \ar[d] &A\ar[r] \ar[d]_f &V^A \ar[d] \ar[r] &U^A \ar@{=}[d]\\
U_B \ar[r] &V_B \ar[r] &B \ar[r]_g &K \ar[r] &U^A.
}
$$
We only prove (a), since (b) is similar.\\
(a) If $\overline f$ is a monomorphism in $\B/\V$, then $\overline {g'}=0$, which implies $g'$ factors through $\V$. If $B\in \U$, we have an exact sequence $H(K') \xrightarrow{H(g')=0} H(A) \xrightarrow{H(f)} H(B)=0$, which implies $H(A)=0$, hence $A\in \K$. Since $\K=\U+\V$ and $A$ has no direct summand in $\V$, we get $A\in \U$.
\end{proof}

\begin{prop}
 If $\B/\V$ is abelian, then $\U/\C$ is abelian.
\end{prop}

\begin{proof}
Note that a morphism in $\U$ factors through $\V$ if and only if it factors through $\C$. Since $\B/\V$ is abelian, then by Lemma \ref{serre}, $\U/\C$ has kernels and cokernals. Now it is enough to show any monomorphism in $\U/\C$ is also a monomorphism in $\B/\V$, the case for epimorphism is by dual.

Let $\overline k:U_1\to U_2$ be a monomorphism in $\U/\C$, it has a kernel $\overline f:X\to U_1$ in $\B/\V$. By Lemma \ref{serre}, $X\in \U$, then $\overline {kf}=0$ implies $\overline f=0$, which means $\overline k$ is a kernel in $\U/\C$. Then $\overline k$ is the kernel of its cokernel $\overline l:U_2\to Y$ in $\B/\V$. By Lemma \ref{serre}, $Y\in \U$, hence $\overline k$ is the kernel some morphism in $\U/\C$. This shows that $\U/\C$ is abelian. 
\end{proof}

\begin{rem}
Since $\U/\C$ is also triangulated, it is semi-simple, which means any monomorphism is a section and any epimorphism is a retraction.
\end{rem}

\begin{cor}\label{zero-iso}
 If $\B/\V$ is abelian, assume that $f: U_1\to U_2$ is a morphism where $U_1,U_2$ are indecomposable objects in $\U$, then $f$ is an isomorphism in $\B$ or factors through $\C$.
\end{cor}

\begin{proof}
If $\overline f\neq 0$, then $\overline f=\overline {hg}$ where $0 \neq \overline g\colon U_1\to B$ is an epimorphism and $0\neq \overline h\colon B\to U_2$ is a monomorphism. By Lemma \ref{serre}, $B$ lies in $\U$. Since $\U/\C$ is semi-simple, $\overline g$ splits, then $B\simeq U_1$ in $\B/\V$. By the same method we can get $B\simeq U_2$. Hence $\overline f:U_1\xrightarrow{\simeq} U_2$ in $\B/\V$. Since $U_1,U_2$ are indecomposable, then $f$ is an isomorphism in $\B$.
\end{proof}

Let $\widetilde{\K}$ (resp. $\widetilde{\U}$, $\widetilde{\V}$) be subcategory of objects which do not have direct summand in $\C$.

\begin{lem}\label{app}
Let $\D$ be a rigid subcategory, $\xymatrix{D \ar[r]^s &C \ar[r]^t &K \ar@{-->}[r] &}$ be an $\EE$-triangle where $K\in \widetilde{\K}$ and $t$ be a right $\C$-approximation (resp. $\mathcal P$ or $\mathcal I$-approximation). Then
\begin{itemize}
\item[(a)] If $K$ is indecomposable, we have $D=C_0\oplus X$ where $C_0\in \mathcal C$ (resp. $\mathcal P$ or $\mathcal I$) and $X$ is indecomposable and $X$ does not belong to $\C$. Moreover, if $t$ is right minimal, then $D=X$.
\item[(b)] $K$ is indecomposable if $D$ is indecomposable.
\end{itemize}
\end{lem}

\begin{proof}
We only show the case when $C\in \C$, the others are similar. Since $K\in \widetilde{\K}$, $D$ does not belong to $\C$.

(a) Let $X$ be an indecomposable direct summand of $D$ that does not belong to $\C$, we have the following commutative diagram
$$\xymatrix{
X \ar[r]^-{x} \ar[d]_-{\alpha} &C \ar@{=}[d] \ar[r]^c &K' \ar[d]^a \ar@{-->}[r] &\\
D \ar[r]^s \ar[d]_-{\beta} &C \ar[d]^f \ar[r]^t &K \ar[d]^b \ar@{-->}[r] &\\
X \ar[r]^-{x} &C \ar[r]^c &K' \ar@{-->}[r] &.
}
$$
where $\beta\alpha=1_{X}$. This diagram implies that $K'$ is a direct summand of $K\oplus C$. If $K'$ is also a direct summand of $C$, the first row splits, hence $X\in \C$, a contradiction. This means $K'=K\oplus C'$ where $C'\in \C$. Hence we have the following commutative diagram
$$\xymatrix{
D \ar[r]^s \ar[d]_{\beta'} &C \ar[r]^t \ar[d]^c &K \ar[d]^-{\svecv{1_K}{0}} \ar@{-->}[r] &\\
X \ar[r]^-{x} \ar[d]_-{\alpha'} &C \ar@{=}[d] \ar[r]^c &K\oplus C' \ar[d]^-{\svech{1_K}{0}} \ar@{-->}[r] &\\
D \ar[r]^s \ &C \ar[r]^t &K \ar@{-->}[r] &
}
$$
This diagram implies that $D$ is a direct summand of $C\oplus X$. Since $D$ does not belong to $\C$, we have $D=C_0\oplus X$.

If $t$ is right minimal, from the diagram above get $c$ is an isomorphism, hence $\alpha'\beta'$ is also an isomorphism. Since $X$ is indecomposable, $D$ is also indecomposable and $D=X$.

\medskip

(b) If $t=0$, we have that $C$ is a direct summand of $D$, hence $D\simeq C\in\C$, a contradiction. We can assume $$C\xrightarrow{t=\svecv{t_1}{t_2}} K_1\oplus K_2=K$$ where $K_1$ is indecomposable and $t_1\neq 0$, then we have the following commutative diagram
$$\xymatrix{
D \ar[r]^-{s} \ar[d]_-{\alpha_1} &C \ar@{=}[d] \ar[r]^-{\svecv{t_1}{t_2}} &K_1\oplus K_2 \ar[d]^-{\svech{1_{K_1}}{0}} \ar@{-->}[r] &\\
D_1 \ar[r] \ar[d]_{\beta_1} \ &C \ar[r]^{t_1} \ar[d] &K_1 \ar[d]^-{\svecv{1_{K_1}}{0}} \ar@{-->}[r] &\\
D \ar[r]^-{s} &C \ar[r]^-{\svecv{t_1}{t_2}} &K_1\oplus K_2 \ar@{-->}[r] &
}
$$
If $\beta_1\alpha_1$ is invertible, then $D_1\simeq D\oplus D_0$. By (a) we have $D_1\simeq X_1\oplus C_0$ where $X$ is indecomposable and $C_0\in \C$, hence $D$ is a direct summand of $X_1\oplus C_0$, which implies $X_1\simeq D$. Then $D_1\simeq D\oplus C_0$. Now we have the following commutative diagram.
$$\xymatrix{
D \oplus C_0 \ar[r]^-{\left(\begin{smallmatrix}
s&0\\
0&1_{C_0}
\end{smallmatrix}\right)} \ar[d]_{\gamma_1}^{\simeq} &C\oplus C_0 \ar[r]^-{\left(\begin{smallmatrix}
t_1&0\\
t_2&0
\end{smallmatrix}\right)} \ar[d] &K_1\oplus K_2 \ar[d]^-{\eta_1} \ar@{-->}[r] &\\
D_1 \ar[r] \ar[d]_{\gamma_2}^{\simeq} \ &C \ar[r]^{t_1} \ar[d] &K_1 \ar[d]^-{\eta_2} \ar@{-->}[r] &\\
D \oplus C_0 \ar[r]_-{\left(\begin{smallmatrix}
s&0\\
0&1_{C_0}
\end{smallmatrix}\right)}  &C\oplus C_0 \ar[r]_-{\left(\begin{smallmatrix}
t_1&0\\
t_2&0
\end{smallmatrix}\right)}  &K_1\oplus K_2 \ar@{-->}[r] &
}
$$
where $\gamma_2\gamma_1=1_{D\oplus C_0}$, then $1_K-\eta_2\eta_1$ factors through $C\oplus C_0$, hence $K$ is a direct summand of $K_1\oplus C\oplus C_0$. Since $K\in \widetilde{\K}$, we have $K=K_1$, hence it is indecomposable.
If $\beta_1\alpha_1$ is not invertible, since $\End_{\B}(D)$ is local, there is a natural number $n$ such that $(\beta_1\alpha_1)^n=0$, this implies $${0=\left(\begin{smallmatrix}
1_{K_1}&0\\
0&0
\end{smallmatrix}\right)}^n\svecv{t_1}{t_2}={\left(\begin{smallmatrix}
1_{K_1}&0\\
0&0
\end{smallmatrix}\right)}\svecv{t_1}{t_2}=\svecv{t_1}{0},$$ hence $t_1=0$, a contradiction.
\end{proof}

The following lemma plays an important role in the rest of the paper.

\begin{lem}\label{mainlem}
If $\B/\V$ is abelian and $\widetilde{\U}\neq 0$, then any indecomposable object $K\in \widetilde{\U}$ admits the following two $\EE$-triangles:
\begin{itemize}
\item[(a)] $\xymatrix{K' \ar[r] &P \ar[r] &K \ar@{-->}[r] &}$,
\item[(b)] $\xymatrix{K \ar[r] &I \ar[r] &K'' \ar@{-->}[r] &}$
\end{itemize}
where $P,I \in \mathcal P \cap \mathcal I$ and $K'$ and $K''\in \widetilde{\U}$ are indecomposables.
\end{lem}

\begin{proof}
Let $A\in \widetilde{\U}$ be an indecomposable object, it admits an $\EE$-triangle $\xymatrix{K_A \ar[r] &C_A \ar[r] &A \ar@{-->}[r] &}$ where $C_A\in \C$ and $K_A\in \U$, by assumption we get $K_A\notin \C$. Let $K\in \widetilde{\U}$ be an indecomposable direct summand of $K_A$, let $K_A=K\oplus K_1$. Then we have the following commutative diagram.
$$\xymatrix{
K_A \ar@{=}[d] \ar[r] &C_A \ar[d] \ar[r] &A \ar[d]^{a_1} \ar@{-->}[r] &\\
K\oplus K_1 \ar[d]_{\svech{1}{0}} \ar[r] &I_{K_A} \ar[d] \ar[r] &\Sigma K_A \ar[d]^{a_2} \ar@{-->}[r] &\\
K \ar[r] & I_K \ar[r] &\Sigma K \ar@{-->}[r] &
}
$$
where $\beta$ is split epimorphism and $I_{K_A},I_K\in \mathcal I$. By Lemma \ref{app}, $K$ admits an $\EE$-triangle $$\xymatrix{
K'\oplus C_1' \ar[r]  &C'\ar[r]  &K  \ar@{-->}[r] &}
$$
where $C',C_1'\in \C$ and $K'\in \widetilde{\U}$ is indecomposable. By the dual of Lemma \ref{app}, $K'$ admits an $\EE$-triangle $\xymatrix{K' \ar[r] &I_{K'} \ar[r]  &X\oplus I' \ar@{-->}[r] &}$ where $X$ is indecomposable and $I_{K'},I'\in \mathcal I$. Then X admits an $\EE$-triangle $\xymatrix{K_X \ar[r] &I_{K'} \ar[r]  &X \ar@{-->}[r] &}$. Hence we have the following commutative diagram
$$\xymatrix{
K_X \ar[r] \ar[d] &I_{K'} \ar[r] \ar[d] &X \ar[d]^{\svecv{1}{0}} \ar@{-->}[r] &\\
K' \ar[r] \ar[d] &I_{K'} \ar[r] \ar@{=}[d] &X\oplus I' \ar[d]^{\svech{1}{0}} \ar@{-->}[r] &\\
K_X \ar[r] &I_{K'} \ar[r]^{j_X} &X \ar@{-->}[r] &
}
$$
Now $K_X$ is a direct summand of $K'\oplus I_{C'}$, hence $K_X\in \U$.
Since $C_1'$ also admits an $\EE$-triangle $\xymatrix{C_1' \ar[r] &I_1'\ar[r] &\Sigma C_1'\ar@{-->}[r] &}$ where $I_1'\in \mathcal I$, let $I'\oplus \Sigma C_1'=D$,then we get the following commutative diagram
$$\xymatrix{
K'\oplus C_1' \ar@{=}[d] \ar[r]  &C'\ar[r] \ar[d] &K \ar[d]^{\svecv{k_X}{k_2}}  \ar@{-->}[r] &\\
K'\oplus C_1' \ar[r] &I_{K'}\oplus I_1' \ar[r] &X\oplus D\ar@{-->}[r] &
}
$$
where morphism $k_X$ does not factors through $\mathcal I$, otherwise $\svecv{k_1}{k_2}$ factors through $\mathcal I$ and then $K'$ is a direct summand of $C'$, a contradiction. Then we have the following commutative diagram
$$\xymatrix{
K_A \ar@{=}[d] \ar[r] &C_A \ar[d] \ar[r] &A \ar[d]^{a_1} \ar@{-->}[r]^-{\delta} &\\
K\oplus K_1 \ar[d]_{\svech{1}{0}} \ar[r] &I_{K_A} \ar[d] \ar[r] &\Sigma K_A \ar[d]^{a_2} \ar@{-->}[r] &\\
K \ar[r] \ar[d]_{k_X} & I_K \ar[r] \ar[d] &\Sigma K \ar[d]^{a_3} \ar@{-->}[r] &\\
X \ar[r] &I_X \ar[r]^{j} &\Sigma X \ar@{-->}[r]^-{\sigma} &
}
$$
where $I_X\in \mathcal I$. Let $a=a_3a_2a_1$, we claim that $a$ does not factors through $j$. Otherwise, $a$ factors through $j$, then $(k_X,0)$ factors through $C_A$. Hence it factors through $j_X$, this means $k_1$ factors through $\mathcal I$, a contradiction.\\
Since $a^*\sigma={k_x}_*\delta$, we have the following commutative diagram
$$\xymatrix{
K_A \ar[d]_{k_X} \ar[r] &C_A \ar[r]^{c_A} \ar[d]_c &A \ar@{=}[d] \ar@{-->}[r]^-{\delta} &\\
X \ar[r]^g \ar@{=}[d] &B \ar[d] \ar[r]^f &A \ar[d]^a \ar@{-->}[r]^-{a^*\sigma} & \\
X \ar[r] &I_X \ar[r]_{j} &\Sigma X \ar@{-->}[r]^-{\sigma} &
}
$$
(I) We show $B\in \C$.\\
Let $\overline f: B\xrightarrow{\overline l} D\xrightarrow{\overline k} A$ be an epic-monic factorization of $\overline f$ in $\B/\V$, then we have $f-kl:B\xrightarrow{v_1} V_A\xrightarrow{v_2} A$ where $V_A\in \V$. Since $\E(V_A,K_A)=0$, we have $v_2:V_A\xrightarrow{v_3} C_A\xrightarrow{c_A}$. Hence $f-lk=v_2v_1=c_Av_3v_1$. By Lemma \ref{serre}, $D\in \U$. Then by Corollary \ref{zero-iso}, $D\in \C$ or $D\simeq A$.

If $D\simeq A$, the $\overline f$ is an epimorphism. If $\overline f$ is also a monomorphism, it is an isomorphism since $\B/\V$ is abelian. Since $A\in \widetilde{\U}$ is indecomposable, we get $f$ is a split epimorphism, which implies $a$ factors through $j$, a contradiction. Let $\overline r:R\to B$ be the kernel of $\overline f$, since $\overline {fg}=0$, there is a morphism $s:X\to R$ such that $\overline {rs}=\overline g$.
On the other hand, we have the following commutative diagram
$$\xymatrix{
&R \ar[r]^{c_1} \ar[d]_r &V \ar[d]^{c_2}\\
X \ar[r]_g &B \ar[r]_f &A \ar@{-->}[r] &
}
$$
where $V\in \V$. Since $\EE(V,K_A)=0$, we have $c_2:V\xrightarrow{c_3} C_A \xrightarrow{c_A} A$. Hence $f(r-cc_3c_1)=0$, then there is a morphism $t:R\to X$ such that $\overline {gt}=\overline r$. Hence $\overline {rst}=\overline {gt}=\overline {r}$, which implies $\overline {st}=\overline {1_{R}}$. Since $X$ is indecomposable, $\overline s$ is an isomorphism, hence $\overline g$ is a monomorphism. This implies $\overline {k_X}=0$, $k_X$ factors through $\V$, then it factors through $C_A$. Hence $A$ is a direct summand of $B$, which implies $a$ factors through $j$, a contradiction.

If $D\in \C$, then we have $k:D\xrightarrow{d} C_A\xrightarrow{c_A} A$, hence we have a morphism $cd:D\to B$ such that $fcd=k$. This implies $f(1_B-cdl-cv_3v_1)=0$, hence there is a morphism $m:B\to X$ such that $gm=1_B-cdl-cv_3v_1$. Hence $\overline {1_B}=\overline {gm}$ and $\overline g$ is a split epimorphism. Since $X$ is indecomposable, $\overline g:X\xrightarrow{\simeq} B$ in $\B/\V$ if $B$ does not belong to $\V$. This means $g$ is a split monomorphism. Then we still have $A$ is direct summand of $B$ and then $a$ factors through $j$, a contradiction. Hence $B\in \V$. From the following commutative diagram
$$\xymatrix{
K_A \ar[d]_{k_X} \ar[r] &C_A \ar[r]^{c_A} \ar[d] &A \ar@{=}[d] \ar@{-->}[r] &\\
X \ar[r]^g \ar[d] &B \ar[d] \ar[r]^f &A \ar@{=}[d] \ar@{-->}[r] & \\
K_A  \ar[r] &C_A \ar[r]_{c_A}  &A \ar@{-->}[r] &
}
$$
we get $K_A$ is a direct summand of $X\oplus C_A$, since $K_A\notin \C$, $X$ is direct summand of $K_A$, hence $X\in \U$ and $B\in \U\cap \V=\C$.\\[2mm]
(II) We show $B\in \mathcal P$.

If $f$ factors through an object $P\in \mathcal P$, then by Remark \ref{useful}, $B$ is a direct summand of $X\oplus P$, since $X\notin \C$, $B$ is a direct summand of $P$, hence $B\in \mathcal P$. We will show $f$ factors through $\mathcal P$.

For an object $B$, there exists an $\EE$-triangle $\xymatrix{\Omega B \ar[r] &P_B \ar[r] &B \ar@{-->}[r]^-{\delta_B} &}$ where $P_B\in \mathcal P$, then we have the following commutative diagram
$$\xymatrix{
&\Omega B \ar@{=}[r] \ar[d]_{y_1} &\Omega B \ar@{=}[r] \ar[d]^{f'} &\Omega B \ar[d]\\
K_X \ar[r]^y \ar@{=}[d] &Y \ar[r]^{y_2} \ar[d]_{y_3} &\Omega A \ar[r]^p \ar[d]^{h'} &P_B \ar[d] \ar[r] &A \ar@{=}[d] \ar@{-->}[r]^-{\delta_A} &\\
K_X \ar[r]^{i_X} &I_{C'} \ar[r] &X \ar[r] &B \ar@{-->}[d]^-{\delta_B} \ar[r]^f &A\\
&&&&&
}
$$
where $f'_*\delta_B=f^*\delta_A$. Hence we have the following diagram
$$\xymatrix{
\Omega B \ar[r]^{pf'} \ar[d]_{f'} &P_B \ar[d] \ar[r] &B \ar[d]^f \ar@{-->}[r]^-{\delta_B} &\\
\Omega A \ar[r]  &P_B  \ar[r] &A \ar@{-->}[r]^-{\delta_A} &
}
$$
Moreover, $\overline {y_1}$ is an epimorphism and by Lemma \ref{app}, $\Omega A$ is indecomposable in $\B/\V$.

If $\overline {f'}=0$, then $f'$ factors through $pf'$, hence $f$ factors though $P_B$.

If $\overline {f'}\neq 0$, then $\overline {y_2}\neq 0$. Consider its epic-monic factorization $Y \xrightarrow{\overline e} E\xrightarrow{\overline m} \Omega A$. We have $\overline {ey}=0$, then there is an object $V_Y\in \V$ such that $ey:K_X\xrightarrow{s_1} V_Y \xrightarrow{s_2} E$. Since $X\in \U$, $\EE(X,V_Y)=0$, then there is a morphism $i:I_{C'}\to V_Y$ such that $s_1=ii_X$, we can replace $e$ by $e-s_2iy_3=e'$. then $e'y=0$, hence there is a morphism $m':\Omega A \to E$ such that $m'y_2=e'$. Since $\overline {m'me}=\overline {e}$, we have $\overline {m'm}=\overline {1_E}$. Since the endomorphism (in $\B/\V$) algebra of $\Omega A$ is local, then either $\overline {mm'}$ or $\overline{1_{\Omega A}}-\overline {mm'}$ is invertible. In the first case $\overline m$ is invertible, then $\overline {f'}$ is epic, hence $\overline {h'}=0$. This implies $h'$ factor through $P_B$, then A is a direct summand of $B$, a contradiction. In the second case, we have $(\overline {1_{\Omega A}}-\overline {mm'})\overline{f'}=0$, hence $\overline{f'}=0$, a contradiction. This implies $f$ factors through $\mathcal P$.

Now we prove that an indecomposable object $K\in \widetilde{\U}$ admits an $\EE$-triangles $\xymatrix{K' \ar[r] &P \ar[r] &K \ar@{-->}[r] &}$ where $P\in \mathcal P$ and $K'\in \widetilde{\U}$. Dually we can show it admits an $\EE$-triangle $\xymatrix{K \ar[r] &I \ar[r] &K'' \ar@{-->}[r] &}$ where $I\in \mathcal I$ and $K'' \in \widetilde{\U}$.

Now by this argument $K'$ also admits an $\EE$-triangles $\xymatrix{K' \ar[r] &I' \ar[r] &K''' \ar@{-->}[r] &}$ where $I'\in \mathcal I$ and $K''' \in \widetilde{\U}$. Then we have the following commutative diagram
$$\xymatrix{
K' \ar[r] \ar@{=}[d] &P \ar[d] \ar[r] &K \ar[d] \ar@{-->}[r] &\\
K' \ar[r] \ar@{=}[d] &I' \ar[r] \ar[d] &K''' \ar[d] \ar@{-->}[r] &\\
K' \ar[r] &P \ar[r] &K \ar@{-->}[r] &
}
$$
Then $P$ is a direct summand of $K'\oplus I'$, since $K\in \widetilde{\U}$, $P$ is a direct summand of $I'$. Hence $P\in \mathcal P\cap \mathcal I$. Dually we have $I\in \mathcal P\cap \mathcal I$.
\end{proof}

We can prove the same results for $\V$, hence we have the following corollary.

\begin{cor}\label{maincor}
If $\B/\V$ and $\B/\U$ are abelian and $\widetilde{\K}\neq 0$, then any indecomposable object $K\in \widetilde{\K}$ admits the following two $\EE$-triangles:
\begin{itemize}
\item[(a)] $\xymatrix{K' \ar[r] &P \ar[r] &K \ar@{-->}[r] &}$,
\item[(b)] $\xymatrix{K \ar[r] &I \ar[r] &K'' \ar@{-->}[r] &}$
\end{itemize}
where $P,I \in \mathcal P \cap \mathcal I$ and $K',K''\in \widetilde{\K}$ are indecomposables.
\end{cor}

Let $\widehat{\K}$ be the subcategory of objects which are direct sums of objects in $\widetilde{\K}$ and $\mathcal P\cap \mathcal I$. Let $\h_{\C}=\CoCone(\C,\C)$, note that $\h_{\C}/\C$ is the heart of cotorsion pair $(\C,\K)$. Then we have the following corollary.

\begin{lem}\label{maincor1}
If any indecomposable object $K\in \widetilde{\K}$ admits the following two $\EE$-triangles:
\begin{itemize}
\item[(a)] $\xymatrix{K' \ar[r] &P \ar[r] &K \ar@{-->}[r] &}$,
\item[(b)] $\xymatrix{K \ar[r] &I \ar[r] &K'' \ar@{-->}[r] &}$
\end{itemize}
where $P,I \in \mathcal P \cap \mathcal I$ and $K',K''\in \widetilde{\K}$ are indecomposables, then
$\B/(\mathcal P\cap \mathcal I)=\widehat{\K}/(\mathcal P\cap \mathcal I)\oplus \h_{\C}/(\mathcal P\cap \mathcal I)$.
\end{lem}

\begin{proof}
Let $A,K$ be indecomposable objects where $A\in \h$ and $K\in \widetilde{\K}$, then $A$ admits an $\EE$-triangle $\xymatrix{A \ar[r]^c &C^1 \ar[r] &C^2 \ar@{-->}[r] &}$ where $C^1,C^2\in \C$. $K$ admits $\EE$-triangles $\xymatrix{K' \ar[r] &P \ar[r]^p &K \ar@{-->}[r] &}$ and $\xymatrix{K \ar[r]^i &I \ar[r] &K'' \ar@{-->}[r] &}$ where $P,I\in \mathcal P\cap \mathcal I$, $K,K''\in \widetilde{\K}$.

Let $f:A\to K$ be a morphism, then there is a morphism $r:C^1\to K$ such that $f=rc$, $c$ factors through $p$, hence $\Hom_{\B/(\mathcal P\cap \mathcal I)}(\h_{\C},\widehat{\K})=0$.

Let $g:K\to A$ be a morphism, since we have the following commutative diagram
$$\xymatrix{
\Omega C^1 \ar@{=}[r] \ar[d] &\Omega C^1 \ar[d]\\
\Omega C^2 \ar[d] \ar[r] &P \ar[r]^p \ar[d] &C^2 \ar@{=}[d] \ar@{-->}[r]&\\
A \ar[r]^c \ar@{-->}[d] &C^1 \ar[r] \ar@{-->}[d] &C^2 \ar@{-->}[r] &\\
&& &&
}
$$
Then we have the following commutative diagram
$$\xymatrix{
K \ar[ddr]_g \ar@{.>}[dr]^3 \ar[rrr]^-i &&& I \ar@{.>}[ddl]^(.25)1 \ar@{.>}[dl]_2 \ar[r] &K'' \ar@{.>}[d]_{4=c'}\\
&\Omega C^2 \ar[d] \ar[r] &P \ar[d] \ar[rr]_p &&C^2\\
&A \ar[r]^c  &C^1 \\
}
$$
The numbers for the morphisms denote the order that we get them. Since $K''$ also admits an $\EE$-triangle
$$\xymatrix{K'' \ar[r]^{i''} &I'' \ar[r] &K''' \ar@{-->}[r] &}$$
where $I''\in \mathcal P\cap \mathcal I$ and $K'''\in \widetilde{\K}$, we get $c'$ factors through $I''$, then $c'$ factors through $p$, hence  morphism $3$ factors through $I$, then $g$ factors through $I$. This means $\Hom_{\B/(\mathcal P\cap \mathcal I)}(\widehat{\K},\h_{\C})=0$.

Assume that $X$ is an indecomposable object, it admits an $\EE$-triangle $$\xymatrix{X \ar[r] &K^X\oplus  C^0 \ar[r] &C^X \ar@{-->}[r] &}$$ where $K^X\in \widetilde{\K}$ and $C^0,C^X\in \C$. By Lemma \ref{mainlem}, $K^X$ admits an $\EE$-triangle $$\xymatrix{K_0 \ar[r] &P_0 \ar[r] &K^X \ar@{-->}[r] &}$$ where $K_0\in \widetilde{\K}$ and $P_0\in \mathcal P$. Hence we have the following commutative diagram
$$\xymatrix{
K_0 \ar[d]_{k_0} \ar@{=}[r] &K_0 \ar[d]\\
Y \ar[d] \ar[r] &P_0\oplus C^0 \ar[r] \ar[d] &C^X \ar@{=}[d] \ar@{-->}[r] &\\
X \ar[r] \ar@{-->}[d] &K^X\oplus  C^0 \ar[r] \ar@{-->}[d] &C^X \ar@{-->}[r] &\\
&& &&
}
$$
where $Y\in \h_\C$, we have $k_0: K_0\xrightarrow{j_1} I \xrightarrow{j_2} Y$ where $I \in \mathcal P\cap \mathcal I$, then $j_1$ factors through $k_0$, hence $Y$ is a direct summand of $X\oplus I$. If $Y$ is a direct summand of $I$, we have $X\in \K$. Otherwise $Y=X\oplus I'$. which implies $X\in \h_{\C}$ by Lemma \ref{direct summand}.
\end{proof}

\begin{lem}\label{maincor2}
If $\B/(\mathcal P\cap \mathcal I)=\widehat{\K}/(\mathcal P\cap \mathcal I)\oplus \h_\C/(\mathcal P\cap \mathcal I)$, then $\widehat{\K}/(\mathcal P\cap \mathcal I),\h_{\C}/(\mathcal P\cap \mathcal I)$ are extriangulated subcategories of $\B/(\mathcal P\cap \mathcal I)$.
\end{lem}

\begin{proof}
It is enough to show that $\widehat{\K}$ and $\h_\C$ are extension closed.

Let $\xymatrix{X\ar[r]^{x} &Y\ar[r]^{y} &Z \ar@{-->}[r] &}$ be an $\EE$-triangle where $X,Z\in \widehat{\K}$. We already have $Y\in \K$. If $Y$ has a indecomposable direct summand $C \in \C$. Since $\Hom_{\B/(\mathcal P\cap \mathcal I)}(C,Z)=0$, then $C$ is a direct summand of some $X\oplus I$ where $I\in \mathcal P\cap \mathcal I$. Hence $C\in \mathcal P\cap \mathcal I$.

Now let $\xymatrix{X\ar[r]^{x} &Y\ar[r]^{y} &Z \ar@{-->}[r] &}$ be an $\EE$-triangle where $X,Z\in \h_\C$. We have $Y=Y_1\oplus Y_2$ where $Y_1\in \h_\C$ and $Y_2\in \widetilde{\K}$. Since $\Hom_{\B/(\mathcal P\cap \mathcal I)}(Y_2,Z)=0$, $Y_2$ is a direct summand of some $X\oplus I'$ where $I'\in \mathcal P\cap \mathcal I$. Hence $Y_2=0$ and $Y\in \h_\C$.
\end{proof}

\begin{prop}\label{mainprop}
Let $\C$ be a pre-cluster tilting subcategory of $\B$, then the following statements are equivalent.
\begin{itemize}
\item[(a)] $\C$ is the category of projective-injective objects of $\B$.
\item[(b)] $\K/C=\B/\C$.
\item[(c)] $\h_{\C}=\C$.
\end{itemize}
\end{prop}

\begin{proof}
(b) $\Longrightarrow$ (c): $\h_{\C}\subseteq \K$, and $\h_{\C}\cap \K=\C$, hence $\h_{\C}=\C$.

(c) $\Longrightarrow$ (a): For any object $B$ we have $H(B)\in \h_{\C}/\C=0$, hence $B\in \K$. We know that $\C$ is the category of projective-injective objects of $\K=\B$.

(a) $\Longrightarrow$ (b): By definition $(\C,\B)$ is a cotorsion pair, hence $\B=\K$ and then $\K/\C=\B/\C$.
\end{proof}

\textbf{Now we can prove Theorem \ref{main}.}

\begin{proof}
By Theorem \ref{main1}, $(\C,\K),(\K,\C)$ are cotorsion pairs.

(a) $\Longrightarrow$ (b): when $\C$ is cluster tilting, we have $\C\subseteq \U \subseteq \K=\C$, hence $\C=\U$.

(b) $\Longrightarrow$ (a): we have cotorsion pairs $(\K,\C),(\C,\K)$, hence $\K=\C$, which implies $\C$ is a cluster tilting subcategory of $\B$.

(a) $\Longrightarrow$ (d): Since $\C$ is cluster tilting, then $\K=\C$ and $\U=\C=\V$, by \cite[Theorem 3.2]{LN}, we know that $\B/\U=\B/\V=\B/\C$ are abelian.

(d) $\Longrightarrow$ (a): Assume $\widetilde{\K}\neq 0$, then by Lemma \ref{maincor2}, $\h/(\mathcal P\cap \mathcal I)$ is an extriangulated subcategory of $\B/(\mathcal P\cap \mathcal I)$. Since $\C/(\mathcal P\cap \mathcal I)\neq 0$, by Proposition \ref{mainprop}, $\h/(\mathcal P\cap \mathcal I)$ is non-zero. Now by Lemma \ref{mainlem} and Lemma \ref{maincor2}, $\B/(\mathcal P\cap \mathcal I)$ is a direct sum of two non-zero extriangulated subcategories, a contradiction.

(a) $\Longrightarrow$ (c) is by \cite[Theorem 3.2]{LN}.

(c) $\Longrightarrow$ (a):  let $\V=\C$, the other arguments are the same as ``(d) $\Longrightarrow$ (a)".
\end{proof}

This theorem immediately yields the following important conclusion.

\begin{cor}
Let $\B/(\mathcal P\cap \mathcal I)$ be connected and $\C$ a pre-cluster tilting subcategory of $\B$. Then $\C$ is cluster tilting if and only if $\B/\C$ is abelian.
\end{cor}

\begin{proof}
This follows that Proposition \ref{prop0} and Theorem \ref{main}.
\end{proof}

Theorem \ref{main} generalizes \cite[Theorem 7.3]{B} for the following reason: If $\B$ is a triangulated category with shift functor $[1]$, then $\mathcal P=0=\mathcal I$. The category $\B$ has a Serre functor $\mathbb{S}$, then $\C^{\bot_1}={^{\bot_1}}(\mathbb{S}[-2]\C)$. The subcategory $\C$ of $\B$ is functorially finite and rigid, by Lemma \ref{CP} we have two cotorsion pairs $(\C,\C^{\bot_1})$ and $({^{\bot_1}}\C,\C)$. Note that if $\mathbb{S}\C=\C[2]$, then $\C^{\bot_1}={^{\bot_1}}\C$, which implies $\C$ is pre-cluster tilting.


\begin{thebibliography}{88}

\bibitem[AN]{AN}
N. Abe, H. Nakaoka.
\newblock General heart construction on a triangulated category (II): Associated cohomological functor.
\newblock Appl. Categ. Structures. 20 (2012), no. 2, 162--174.


\bibitem[B]{B}
A. Beligiannis.
\newblock Rigid objects, triangulated subfactors and abelian
localizations.
\newblock Math. Z. 20 (2013), no. 274, 841--883.

\bibitem[DL]{DL} L. Demonet, Y. Liu. Quotients of exact categories by cluster tilting subcategories as module categories.  J. Pure Appl. Algebra 217(12), no. 12,  2282--2297,.

\bibitem[I]{I}
O. Iyama. \newblock Higher-dimensional Auslander-Reiten theory on maximal orthogonal subcategories. \newblock Adv. Math.
210(2007), 22--50.

\bibitem[IY]{IY}
O. Iyama, Y. Yoshino.
\newblock Mutation in triangulated categories and rigid Cohen-Macaulay modules.
\newblock Invent. Math. 172 (2008), no. 1, 117--168.

\bibitem[KZ]{KZ}
S. Koenig, B. Zhu.
\newblock From triangulated categories to abelian categories: cluster tilting in a general framework.
\newblock Math. Z. 258 (2008), no. 1, 143--160.

\bibitem[L]{L}
Y. Liu.
\newblock Hearts of cotorsion pairs are functor categories over cohearts.
\newblock arXiv: 1504.05271.

\bibitem[LN]{LN}
Y. Liu, H. Nakaoka.
\newblock Hearts of twin cotorsion pairs on extriangulated categories.
\newblock arXiv: 1702.00244.


\bibitem[N1]{N1}
H. Nakaoka.
\newblock General heart construction on a triangulated category (I): unifying $t$-structures and cluster tilting subcategories.
\newblock Appl. Categ. Structures 19 (2011), no. 6, 879--899.

\bibitem[NP]{NP}
H. Nakaoka, Y. Palu.
\newblock Mutation via Hovey twin cotorsion pairs and model structures in extriangulated categories.
\newblock arXiv:1605.05607.

\bibitem[ZZ]{ZZ}
P. Zhou, B. Zhu.
\newblock Triangulated quotient categories revisited.
\newblock J. Algebra 502 (2018), 196-232.
\end{thebibliography}
\end{document}